\documentclass[12pt,reqno]{amsart}
\usepackage{amscd,amssymb,amsmath,amsthm}
\usepackage[arrow,matrix]{xy}
\usepackage{graphicx}
\usepackage{cite}
\usepackage{geometry}
\newtheorem{thm}{Theorem}
\newtheorem{defn}{Definition}

\newtheorem{pro}{Proposition}
\newtheorem{rem}{Remark}
\newtheorem{cor}{Corollary}

\numberwithin{equation}{section}

\begin{document}
\title[Normal subgroups]{New normal subgroups for the group representation of the Cayley tree.}

\author{F.H.Haydarov}
 \address{F.H.Haydarov\\ National University of Uzbekistan,
Tashkent, Uzbekistan.}
 \email {haydarov\_ imc@mail.ru.}

\begin{abstract}
In this paper we give full description of normal subgroups of
index eight and ten for the group.
\end{abstract}

\maketitle



\section{Introduction}

 There are several thousand papers and books devoted
 to the theory of groups. But still there are unsolved
 problems, most of which arise in solving of problems of
  natural sciences as physics, biology etc.
 In particular, if configuration of physical system is
 located on a lattice (in our case on the graph of a group)
 then the configuration can be considered as a function defined on
 the lattice. There are many works devoted to several kind of
 partitions of groups (lattices) (see e.g. \cite{1},\cite{3},\cite{5},\cite{7}).

One of the central problems in the theory of Gibbs measures is to
study periodic Gibbs measures corresponding to a given
Hamiltonian. It is known that there exists a one to one
correspondence between the set of vertices $V$ of the Cayley tree
$\Gamma^{k}$ and the group $G_{k}$(see\cite{5}). For any normal
subgroups $H$ of the group $G_{k}$ we define $H$-periodic Gibbs
measures.

 In Chapter 1 of \cite{5} it was constructed several normal subgroups of
the group representation of the Cayley tree. In \cite{4} we found
full description of normal subgroups of index four and six for the
group. In this paper we continue this investigation and construct
all normal subgroups of index eight and ten for the group representation of the Cayley tree.\\

   {\it Cayley tree.}   \,\  A Cayley tree (Bethe lattice) $\Gamma^k$ of order $k\geq 1$ is an infinite
homogeneous tree, i.e., a graph without cycles, such that exactly
$k+1$ edges originate from each vertex. Let $\Gamma^k=(V,L)$ where
$V$ is the set of vertices and $L$ that of edges (arcs).\\

{\it A group representation of the Cayley tree.} \,\ Let $G_{k}$
be a free product of $k+1$ cyclic groups of the second order with
generators $a_{1},a_{2},...a_{k+1},$ respectively.

 It is known that there exists a one to one correspondence between
the set of vertices $V$ of the Cayley tree $\Gamma^{k}$ and the
group $G_{k}$.

To give this correspondence we fix an arbitrary element $x_{0}\in
V$ and let it correspond to the unit element $e$ of the group
$G_{k}.$ Using $a_{1},...,a_{k+1}$ we numerate the
nearest-neighbors of element $e$, moving by positive direction.
Now we'll give numeration of the nearest-neighbors of each $a_{i},
i=1,...,k+1$ by $a_{i}a_{j}, j=1,...,k+1$. Since all $a_{i}$ have
the common neighbor $e$ we give to it $a_{i}a_{i}=a_{i}^{2}=e.$
Other neighbor are numerated starting from $a_{i}a_{i}$ by the
positive direction. We numerate the set of all the
nearest-neighbors of each $a_{i}a_{j}$ by words $a_{i}a_{j}a_{q},
q=1,...,k+1,$ starting from $a_{i}a_{j}a_{j}=a_{i}$ by the
positive direction. Iterating this argument one gets a one-to-one
correspondence between the set of vertices $V$ of the Cayley tree
$\Gamma^{k}$ and the group $G_{k}.$

 Any(minimal represented) element $x\in G_{k}$ has the following form:
\,\ $x=a_{i_{1}}a_{i_{2}}...a_{i_{n}},$ where $1\leq i_{m}\leq
k+1, m=1,...,n.$  The number $n$ is called the length of the word
$x$ and is denoted by $l(x).$  The number of letters $a_{i},
i=1,...,k+1,$ that enter the non-contractible representation of
the word $x$ is denoted by $w_{x}(a_{i}).$

\begin{pro}\label{p1}\cite{2} \ Let $\varphi$ be homomorphism of the group
$G_{k}$ with the kernel $H.$ Then $H$ is a normal subgroup of the
group $G_{k}$ and $\varphi(G_{k})\simeq G_{k}/H,$ (where $G_{k}/H$
is a quotient group) i.e., the index $|G_{k}:H|$ coincides with
the order $|\varphi(G_{k})|$ of the group $\varphi(G_{k}).$
\end{pro}
 Let $H$ be a normal subgroup of a group $G$. Natural
homomorphism $g$ from $G$ onto the quotient group $G/H$ by the
formula $g(a)=aH$ for all $a\in G.$ Then $Ker\varphi=H.$

\begin{defn}\label{d0}\ Let $M_{1}, M_{2},..., M_{m}$ be some sets and
$M_{i}\neq M_{j},$ for $i\neq j.$ We call the intersection
$\cap_{i=1}^{m}M_{i}$ contractible if there exists $i_{0} (1\leq
i_{0}\leq m)$ such that
 $$\cap_{i=1}^{m}M_{i}=\left(\cap_{i=1}^{i_{0}-1}M_{i}\right)\cap\left(\cap_{i=i_{0}+1}^{m}M_{i}\right).$$
\end{defn}

  Let $N_{k}=\{1,...,k+1\}.$ The following Proposition describes
  several normal subgroups of $G_{k}.$

Put \begin{equation} H_{A}=\left\{x\in G_{k}\,\ | \,\ \sum_{i\in
A}\omega_{x}(a_{i}) \ {\it is \ even} \right\},\ \ A\subset N_{k}.
\end{equation}

\begin{pro}\label{p2} \cite{5} For any $\emptyset \neq A \subseteq
  N_{k},$ the set $H_{A}\subset G_{k}$ satisfies the
  following properties:
 (a) $H_{A}$ is a normal subgroup and $|G_{k}:H_{A}|=2;$\\
 (b) $H_{A}\neq H_{B}$,\ for all $A\neq B \subseteq N_{k};$\\
 (c) Let $A_{1}, A_{2},...,A_{m}\subseteq N_{k}.$ If $\cap_{i=1}^{m}H_{A_{i}}$ is non-contractible, then it is
 a  normal subgroup of index $2^{m}.$\end{pro}

\begin{thm}\label{th0} \cite{5}

\item 1. The group $G_{k}$ does not have normal subgroups of odd
index $(\neq 1)$.

\item 2. The group $G_{k}$ has a normal subgroups of arbitrary even
index.

\end{thm}
\section{New normal subgroups of finite index.}

\subsection{The case of index eight.}
\begin{defn}\label{d1} A group $G$ is called a \textbf{dihedral}
group of degree $4$ $(i.e.,D_{4})$ if $G$ is generated by two
elements $a$ and $b$ satisfying the relations
$$o(a)=4, \ o(b)=2, \ ba=a^{3}b.$$
\end{defn}

\begin{defn}\label{d2} A group $G$ is called a
\textbf{quaternion} group $(i.e., Q_{8})$ if $G$ is generated by
two elements $a,\ b$ satisfying the relation
$$o(a)=4, a^{2}=b^{2}, ba=a^{3}b.$$
\end{defn}

\begin{rem}\label{r1}\emph{\cite{2}}  $D_{4}$ is not isomorph to $Q_{8}.$
\end{rem}

\begin{defn}\label{d9} A commutative group $G$ is called a \textbf{Klein 8-group}
 $(i.e.,K_{8})$ if $G$ is generated by three
elements $a, b$ and $c$ satisfying the relations
$o(a)=o(b)=o(c)=2.$
\end{defn}

\begin{pro}\label{p1.}\emph{\cite{2}}  There exist (up to isomorphism)
only two noncommutative nonisomorphic groups of order $8$
\end{pro}

 \begin{pro}\label{p2.} Let $\varphi$ is a homomorphism of the group $G_{k}$ onto a group
 $G$ of order 8. Then
 $\varphi(G_{k})$ is isomorph to either $D_{4}$ or $K_{8}.$
\end{pro}

\begin{proof}

\texttt{Case 1} Let $\varphi(G_{k})$ is isomorph to any
noncommutative group of order $8$. By Proposition \ref{p1}
$\varphi(G_{k})$ is isomorph to either $D_{4}$ or $Q_{8}.$ Let
$\varphi(G_{k})\simeq Q_{8}$ and $e_{q}$ is an identity element of
the group $Q_{8}$. Then
 $e_{q}=\varphi(e)=\varphi(a_{i}^{2})=(\varphi(a_{i}))^{2}$ where
 $a_{i}\in G_{k}, \,\ i\in N_{k}.$
 Hence for the order of $\varphi(a_{i})$ we have $o(\varphi(a_{i}))\in
 \{1,2\}.$ It is easy to check there are only two elements of the
 group $Q_{8}$ which order of element less than two. This is
 contradict.

  \texttt{Case 2} \ Let $\varphi(G_{k},\ast)$ is isomorph to any
commutative group $(G,\ast_{1})$ of order $8$. Then there exist
distinct elements $a,b \in G$ such that $o(a)=o(b)=2.$ Let
$H=\{e,a,b,ab\}.$ It's easy to check that $(H ,\ast_{1})$ is a
normal subgroup of the group $(G,\ast_{1}).$ For $c\notin H$ we
have $H\neq cH \ (cH=c\ast_{1}H).$ Hence $\varphi(G_{k},\ast)$ is
isomorph to only one commutative group $(cH\cup H,\ast_{1}).$
Clearly $(cH\cup H,\ast_{1})\simeq K_{8}.$
\end{proof}

The group $G$ has a finitely generators of the order two and $r$
is a minimal number of such generators of the group $G$ and
without loss of generality we can take these generators are
$b_{1},b_{2},...b_{r}.$ Let $e_{1}$ is an identity element of the
group $G.$ We define homomorphism from $G_{k}$ onto $G.$ Let
$\Xi_{n}=\{A_{1}, \,\ A_{2},...,A_{n}\}$ be a partition of
$N_{k}\backslash A_{0}, \,\ 0\leq|A_{0}|\leq k+1-n.$ Then we
consider homomorphism
$u_{n}:\{a_{1},a_{2},...,a_{k+1}\}\rightarrow \{e_{1},
b_{1}...,b_{m}\}$ as\\
 \begin{equation} u_{n}(x)=\left\{\begin{array}{ll}
e_{1}, \ \ \mbox{if} \ \ x=a_{i}, i\in A_{0}\\[2.5mm]
b_{j}, \ \mbox{if}\ \ x=a_{i}, i\in A_{j}, j=\overline{1,n}.
\\ \end{array}\right.\end{equation}

For $b\in G$ we denote $R_{b}[ b_{1}, b_{2}, ..., b_{m}]$ is a
representation of the word $b$ by generators $b_{1}, b_{2}, ...,
b_{r}, \ r\leq m.$  Define the homomorphism $\gamma_{n}:
G\rightarrow G$ by the formula\\
\begin{equation}\gamma_{n}(x)=\left\{\begin{array}{ll}
e_{1}, \ \ \mbox{if} \ \ x=e_{1}\\[2.5mm]
b_{i}, \ \mbox{if}\ \ x=b_{i}, i=\overline{1,r}\\[2.5mm]
R_{b_{i}}[b_{1},...,b_{r}], \ \mbox{if}\ \ x=b_{i},  i\neq \overline{1,r}\\
\end{array}\right.\end{equation}

Put\\
\begin{equation} H^{(p)}_{\Xi_{n}}(G)=\{x\in G_{k}|\,\
l(\gamma_{n}(u_{n}(x))):2p\},\ 2\leq n \leq k-1. \end{equation}

Let $\gamma_{n}(u_{n}(x)))=\tilde{x}.$ We introduce the following
equivalence relation on the set $G_{k}:$ $x\sim y$ if
$\tilde{x}=\tilde{y}.$ It's easy to check this relation is
reflexive, symmetric and transitive.

\begin{pro}\label{p3.} \ Let \ $\Xi_{n}=\{A_{1}, A_{2},...,A_{n}\}$ be a partition of
$N_{k}\backslash A_{0}, \,\ 0\leq|A_{0}|\leq k+1-n.$ Then
$H^{(p)}_{\Xi_{p}}(G)$ is a normal subgroup of index $2p$ of the
group $G_{k}.$\end{pro}
\begin{proof} \,\  For $x=a_{i_{1}}a_{i_{2}}...a_{i_{n}}\in G_{k}$
it's sufficient to show that $x^{-1}H^{(p)}_{\Xi_{n}}(G)$
$x\subseteq H^{(p)}_{\Xi_{n}}(G).$ Suppose that there exist $y\in
G_{k}$ such that $l(\tilde{y})\geq 2p.$ Let
$\tilde{y}=b_{i_{1}}b_{i_{2}}...b_{i_{q}}, \ q\geq 2p$ and
$S=\{b_{i_{1}}, b_{i_{1}}b_{i_{2}},...,
b_{i_{1}}b_{i_{2}}...b_{i_{q}}\}.$ Since $S\subseteq G$ there
exist $x_{1}, x_{2} \in S$ such that $x_{1}=x_{2}.$ But this is
contradict to $\tilde{y}$ is a non-contractible. Thus we have
proved that $l(\tilde{y})<2p.$ Hence for any $x\in
H^{(p)}_{\Xi_{n}}(G)$ we have $\tilde{x}=e_{1}.$ Now we take any
element $z$ from the set $x^{-1}H^{(p)}_{\Xi_{n}}(G)\ x.$ Then
$z=x^{-1}h\ x$ for some $h\in H^{(p)}_{\Xi_{n}}(G).$ We have
 $\tilde{z}=\gamma_{n}(v_{n}(z))=\gamma_{n}\left(v_{n}(x^{-1}h\ x)\right)=$
 $=\gamma_{n}\left(v_{n}(x^{-1})v_{n}(h)v_{n}(x)\right)=
 \gamma_{n}\left(v_{n}(x^{-1})\right)\gamma_{n}\left(v_{n}(h)\right)\gamma_{n}\left(v_{n}(x)\right).$
From $\gamma_{n}\left(v_{n}(h)\right)=e_{1}$ we get
$\tilde{z}=e_{1}$ i.e., $z\in H^{(p)}_{\Xi_{n}}(G).$ This
completes the proof.\end{proof}

 For $A_{1},A_{2},A_{3}\subset N_{k}$ and $\cap_{i=1}^{3}H_{A_{i}}$
is non-contractible we denote following set\\
$$\Re=\{\cap_{i=1}^{3}H_{A_{i}}|
\ A_{1},A_{2},A_{3}\subset N_{k}\}$$

\begin{thm}\label{th1}  For the group
$G_{k}$ following statement is hold\\
 $$\{H|\ H \ is\ a\ normal\
subgroup\ of\ G_{k}\ with\ |G_{k}:H|=8\}=$$
$$=\{H^{(4)}_{C_{0}C_{1}C_{2}}(D_{4})|\ C_{1},C_{2}\ is\ a\
partition\ of\ N_{k}\setminus C_{0}\} \cup \Re.$$
\end{thm}

 \begin{proof} Let $\phi$ be a
homomorphism with $|G_{k}:Ker\phi|=8.$ Then by Proposition 2 we
have $\phi(G_{k})\simeq K_{8}$ and $\phi(G_{k})\simeq D_{4}.$

 Let $\phi:G_{k}\rightarrow K_{8}$ be an epimorphism. For any
 nonempty sets $A_{1},A_{2},A_{3}\subset N_{k}$ we give one to one
 correspondence between $\{Ker\phi| \ \phi(G_{k})\simeq K_{8}\}$ and
 $\Re.$ Let $a_{i}\in G_{k}, i\in N_{k}.$
We define following homomorphism corresponding to the set $A_{1},
A_{2}, A_{3}.$

$$\phi_{A_{1}A_{2}A_{3}}(a_{i})=\left\{\begin{array}{ll}
a, \ \ \mbox{if} \ \ i\in A_{1}\setminus(A_{2}\cup A_{3})\\[2.5mm]
b, \ \ \mbox{if} \ \ i\in A_{2}\setminus(A_{1}\cup A_{3})\\[2.5mm]
c, \ \ \mbox{if} \ \ i\in A_{3}\setminus(A_{1}\cup A_{2})\\[2.5mm]
ab, \ \ \mbox{if} \ \ i\in (A_{1}\cap A_{2})\setminus(A_{1}\cap A_{2}\cap A_{3})\\[2.5mm]
ac, \ \ \mbox{if} \ \ i\in (A_{1}\cap A_{3})\setminus(A_{1}\cap A_{2}\cap A_{3})\\[2.5mm]
bc, \ \ \mbox{if} \ \ i\in (A_{2}\cap A_{3})\setminus(A_{1}\cap A_{2}\cap A_{3})\\[2.5mm]
abc, \ \ \mbox{if} \ \ i\in A_{1}\cap A_{2}\cap A_{3}\\[2.5mm]
e, \ \ \mbox{if} \ \ i\in N_{k}\setminus(A_{1}\cup A_{2}\cup
A_{3}). \\
\end{array}\right.$$

If $i\in \emptyset$ then we'll accept that there is not any index
$i\in N_{k}$ which that condition is not satisfied. It is easy to
check $Ker \phi_{A_{1}A_{2}A_{3}}=H_{A_{1}}\cap H_{A_{2}}\cap
H_{A_{3}}.$ Hence $\{Ker\phi| \ \phi(G_{k})\simeq K_{8}\}=\Re.$

Now we'll consider the case $\phi(G_{k})\simeq D_{4}.$ Let $\phi:
G_{k}\rightarrow D_{4}$ be epimorphisms. Put\\
$$C_{0}=\{i| \ \phi(a_{i})=e\}, \ \ C_{1}=\{i| \ \phi(a_{i})=b\},\  C_{2}=\{i| \ \phi(a_{i})=ab\}$$

One can construct following homomorphism (corresponding to $C_{0},
C_{1}, C_{2}$)\\
$$\phi_{C_{0} C_{1}
C_{2}}(x)=\left\{\begin{array}{ll}
e, \ \ \mbox{if} \ \ \tilde{x}=e\\[2.5mm]
a, \ \mbox{if}\ \ \tilde{x}=b_{2}b_{1} \\[2.5mm]
a^{2}, \ \mbox{if} \ \ \tilde{x}=b_{2}b_{1}b_{2}b_{1} \\[2.5mm]
a^{3}, \ \mbox{if} \ \ \tilde{x}=b_{2}b_{1}b_{2}b_{1}b_{2}b_{1}\\[2.5mm]
b, \ \mbox{if}\ \ \tilde{x}=b_{1} \\[2.5mm]
ab, \ \mbox{if}\ \ \tilde{x}=b_{2} \\[2.5mm]
a^{2}b, \ \mbox{if}\ \ \tilde{x}=b_{2}b_{1}b_{2} \\[2.5mm]
a^{3}b, \ \mbox{if}\ \ \tilde{x}=b_{2}b_{1}b_{2}b_{1}b_{2}. \\
\end{array}\right.$$

Straight away we conclude $Ker(\phi_{C_{0} C_{1}
C_{2}})=H^{(4)}_{C_{0} C_{1} C_{2}}(D_{4}).$ We have constructed
all homomorphisms $\phi$ on the group $G_{k}$ which
$|G_{k}:Ker\phi|=8.$ Thus by Proposition \ref{p1} one get\\
 $$\{H|\ |G_{k}:H|=8\}\subseteq \{H^{(4)}_{C_{0}C_{1}C_{2}}(D_{4})|\
C_{1},C_{2}\ is\ a\ partition\ of\ N_{k}\setminus C_{0}\} \cup
\Re.$$\\
 By Proposition \ref{p2} and Proposition \ref{p3.} we can see
 easily\\
$$\Re \cup \{H^{(4)}_{C_{0}C_{1}C_{2}}(D_{4})|\ C_{1},C_{2}\
is\ a\ partition\ of\ N_{k}\setminus C_{0}\}\subseteq \{H|\
|G_{k}:H|=8\}.$$\\
 The theorem is proved.
 \end{proof}

\begin{cor}\label{c1} The number of all normal subgroups of index
$8$ for the group $G_{k}$ is equal to
$8^{k+1}-6\cdot4^{k+1}+3^{k+1}+9\cdot2^{k+1}-5.$
\end{cor}

\begin{proof} Number of elements of the set $H_{A}\subset G_{k}, \emptyset \neq A\subset
N_{k}$ is $2^{k+1}-1.$ Then
$|\Re|=(2^{k+1})(2^{k+1}-2)(2^{k+2}-3).$ Let $C_{0}\subset N_{k}$
is a fixed set and $|C_{0}|=p.$ If $C_{1}, C_{2}$ is a partition
of $N_{k}\setminus C_{0}$ then there are $2^{k-p+1}-2$ ways to
choose the sets $C_{1}$ and $C_{2}.$ Hence the cardinality of
$\{H^{(4)}_{C_{0}C_{1}C_{2}}(D_{4})|\ C_{1},C_{2}\ is\ a\
partition\ of\ N_{k}\setminus C_{0}\}$ is equal to\\
$$(2^{k+1}-2)C^{0}_{k+1}+(2^{k}-2)C_{k+1}^{1}+...+2C_{k+1}^{k-1}=3^{k+1}-2^{k+2}+1.$$\\
Since $\Re$ and $\{H^{(4)}_{C_{0}C_{1}C_{2}}(D_{4})\}|\
C_{1},C_{2}\subset N_{k}$ are disjoint sets, cardinality of the
union of these sets is
$8^{k+1}-6\cdot4^{k+1}+3^{k+1}+9\cdot2^{k+1}-5.$

\end{proof}

\subsection{Case of index ten.}

 Let the group $R_{10}$ is generated by the permutations
$$\pi_{1}=(1,2)(3,4)(5,6), \ \pi_{2}=(2,3)(4,5)$$

\begin{pro}\label{p2.1} Let $\varphi$ is a homomorphism of the group $G_{k}$ onto a group $G$ of order 10.
 Then
 $\varphi(G_{k})$ is isomorph to $R_{10}.$
\end{pro}

\begin{proof} Let $(G,\ast)$ be a group and $|G|=10.$ Suppose
there exist an epimorphism from $G_{k}$ onto $G.$ It is easy to
check that there are at least two elements $a,b\in G_{k}$ such
that $o(a)=o(b)=2.$ If $a\ast b=b\ast a$ then $(H,\ast)$ is a
subgroup of the group $(G,\ast),$ where $H=\{e,a,b,a\ast b\}.$
Then by Lagrange's theorem $|G|$ is divide by $|H|$ but $10$ is
not divide by $4.$ Hence $a\ast b\neq b\ast a.$ We have
$\{e,a,b,a\ast b,b\ast a\}\subset G$ If $G$ is generated by three
elements then there exist an element $c$ such that $c\notin
\{e,a,b,a\ast b,b\ast a\}.$ Then the set $\{e,a,b,a\ast b, b\ast
a, c,c\ast a, c\ast b, c\ast a\ast b, c\ast b\ast a\}$ must be
equal to $G.$ Since $G$ is a group we get $a\ast b\ast a=b$ but
from $o(a)=2$ the last equality is equivalent to $a\ast b=b\ast
a.$ This is a contradict. Hence by Lagrange is theorem it's easy
to see
$$G=\{e,a,b,a\ast b, b\ast a, a\ast b\ast a, b\ast a\ast b,a\ast
b\ast a\ast b,b\ast a\ast b\ast a,a\ast b\ast a\ast b\ast a\},$$
where $o(a\ast b)=5.$ Namely $G\simeq R_{10}.$ This completes the
proof.
\end{proof}

\begin{thm}\label{th2} \  For the group
$G_{k}$ following statement is hold
$$\{H|\ |\ H\ is\ a\ normal\ subgroup\ of\ G_{k}\ with\ |G_{k}:H|=10\}=$$
$$=\{H^{(5)}_{B_{0}B_{1}B_{2}}(R_{10})|\ B_{1},
B_{2} \ is\ a\ partition\ of\ the\ set\ N_{k}\setminus B_{0}\}.$$
\end{thm}

\begin{proof} Let $\phi$ be a homomorphism with
$|G_{k}:Ker\phi|=10.$ By Proposition \ref{p2.1} $\phi(G_{k})\simeq
R_{10}$ and by Proposition \ref{p3.} we can see easily
$$\{H^{(5)}_{B_{0}B_{1}B_{2}}(R_{10})|\ B_{1},
B_{2} \ is\ a\ partition\ of\ the\ set\ N_{k}\setminus
B_{0}\}\subset \{H|\ |G_{k}:H|=10\}.$$
 Let $\varphi: G_{k}\rightarrow R_{10}$ be epimorphisms. Denote\\
$$B_{0}=\{i| \ \varphi(a_{i})=e\}, \ \ B_{1}=\{i| \ \varphi(a_{i})=a,\  B_{2}=\{i| \
\varphi(a_{i})=b\}.$$\\
Then we can show this homomorphism (corresponding to $B_{1},\
B_{2},\ B_{3}$), i.e.,\\
$$\phi_{B_{0}B_{1}B_{2}}(x)=\left\{\begin{array}{ll}
e, \ \ \mbox{if} \ \ \tilde{x}=e\\[2.5mm]
a, \ \mbox{if}\ \ \tilde{x}=b_{1} \\[2.5mm]
b, \ \mbox{if} \ \ \tilde{x}=b_{2} \\[2.5mm]
a\ast b, \ \mbox{if} \ \ \tilde{x}=b_{1}b_{2}\\[2.5mm]
b\ast a, \ \mbox{if}\ \ \tilde{x}=b_{2}b_{1} \\[2.5mm]
a\ast b\ast a, \ \mbox{if}\ \ \tilde{x}=b_{1}b_{2}b_{1} \\[2.5mm]
b\ast a\ast b, \ \mbox{if}\ \ \tilde{x}=b_{2}b_{1}b_{2} \\[2.5mm]
a\ast b\ast a\ast b, \ \mbox{if}\ \ \tilde{x}=b_{1}b_{2}b_{1}b_{2} \\[2.5mm]
b\ast a\ast b\ast a, \ \mbox{if}\ \ \tilde{x}=b_{2}b_{1}b_{2}b_{1} \\[2.5mm]
a\ast b\ast a\ast b\ast a, \ \mbox{if}\ \ \tilde{x}=b_{1}b_{2}b_{1}b_{2}b_{1}. \\
\end{array}\right.$$

 We have constructed all homomorphisms $\phi$ on the group
$G_{k}$ which $|G_{k}:Ker\phi|=10.$ Hence\\
$$\{Ker\phi|\ |G_{k}:Ker\phi|=10\}\subset \{H^{(5)}_{B_{0}B_{1}B_{2}}(R_{10})|\ B_{1},
B_{2} \ is\ a\ partition\ of\ the\ set\ N_{k}\setminus
B_{0}\}.$$\\
By Proposition \ref{p1}\\
$$\{H|\ |G_{k}:H|=10\}=\{H^{(5)}_{B_{0}B_{1}B_{2}}(R_{10})|\ B_{1},
B_{2} \ is\ a\ partition\ of\ the\ set\ N_{k}\setminus
B_{0}\}.$$\\
 The theorem is proved.\end{proof}
\begin{cor}\label{c2} The number of all normal subgroups of
index $10$ for the group $G_{k}$ is equal to $3^{k+1}-2^{k+2}+1.$
\end{cor}

\begin{proof} To prove this Corollary is similar to proof of
Corollary \ref{c1}.
\end{proof}

\section*{ \textbf{Acknowledgements}}
I am deeply grateful to Professor U.A.Rozikov for the attention to
my work.

\end{document}